\documentclass[twoside,12pt]{article}
\usepackage{amsmath,amssymb,fdsymbol}
\usepackage{url}
\begin{document}
\newtheorem{theorem}{Theorem}
\newtheorem{lemma}[theorem]{Lemma}
\newtheorem{corollary}[theorem]{Corollary}
\newtheorem{definition}[theorem]{Definition}
\newtheorem{example}[theorem]{Example}
\pagenumbering{roman}
\renewcommand{\thetheorem}{\thesection.\arabic{theorem}}
\renewcommand{\thelemma}{\thesection.\arabic{lemma}}
\newenvironment{proof}{\noindent{\bf{Proof.\/}}}{\hfill$\blacksquare$\vskip0.1in}
\renewcommand{\thetable}{\thesection.\arabic{table}}
\renewcommand{\thedefinition}{\thesection.\arabic{definition}}
\renewcommand{\theexample}{\thesection.\arabic{example}}
\renewcommand{\theequation}{\thesection.\arabic{equation}}
\newcommand{\mysection}[1]{\section{#1}\setcounter{equation}{0}
\setcounter{theorem}{0} \setcounter{lemma}{0}
\setcounter{definition}{0}}
\newcommand{\mrm}{\mathrm}
\newcommand{\be}{\begin{equation}}
\newcommand{\ee}{\end{equation}}

\newcommand{\ben}{\begin{enumerate}}
\newcommand{\een}{\end{enumerate}}

\title
{\bf Exactness and Convergence Properties of Some Recent  Numerical Quadrature Formulas for Supersingular Integrals of Periodic Functions}

\author
{Avram Sidi\\
Computer Science Department\\
Technion - Israel Institute of Technology\\ Haifa 32000, Israel\\
E-mail:\quad  \url{asidi@cs.technion.ac.il}\\
URL:\quad    \url{http://www.cs.technion.ac.il/~asidi}}
\date{May 2019}
\bigskip\bigskip
\maketitle \thispagestyle{empty}
\newpage\noindent

\begin{abstract}
In a recent work, we developed three new compact numerical quadrature formulas for finite-range periodic
supersingular integrals  $I[f]=\intBar^b_a f(x)\,dx$,  where $f(x)=g(x)/(x-t)^3,$ assuming that $g\in C^\infty[a,b]$ and $f(x)$ is $T$-periodic, $T=b-a$.
With $h=T/n$, these numerical quadrature formulas  read
 \begin{align*}
 \widehat{T}{}^{(0)}_n[f]&=h\sum^{n-1}_{j=1}f(t+jh)
          -\frac{\pi^2}{3}\,g'(t)\,h^{-1}+\frac{1}{6}\,g'''(t)\,h,\\
  \widehat{T}{}^{(1)}_n[f]&=h\sum^n_{j=1}f(t+jh-h/2)
                      -\pi^2\,g'(t)\,h^{-1}, \\
 \widehat{T}{}^{(2)}_n[f]&=2h\sum^n_{j=1}f(t+jh-h/2)-
              \frac{h}{2}\sum^{2n}_{j=1}f(t+jh/2-h/4).
 \end{align*}
 We also  showed that these formulas have spectral accuracy; that is,
 $$\widehat{T}{}^{(s)}_n[f]-I[f]=O(n^{-\mu})\quad\text{as $n\to\infty$}\quad \forall \mu>0.$$
 In  the present work, we continue our study of these  formulas for the special case in which
 $f(x)=\frac{\cos\frac{\pi(x-t)}{T}}{\sin^3\frac{\pi(x-t)}{T}}\,u(x)$, where $u(x)$ is in $C^\infty(\mathbb{R})$ and is $T$-periodic. Actually, we prove that  $\widehat{T}{}^{(s)}_n[f]$,
 $s=0,1,2,$
are exact for a class of singular integrals involving $T$-periodic trigonometric polynomials of degree at most $n-1$; that is,
$$ \widehat{T}{}^{(s)}_n[f]=I[f]\quad\text{when\ \ $f(x)=\frac{\cos\frac{\pi(x-t)}{T}}{\sin^3\frac{\pi(x-t)}{T}}\,\sum^{n-1}_{m=-(n-1)}
c_m\exp(\mrm{i}2m\pi x/T)$.}$$
We also prove that, when $u(z)$ is analytic in a strip
$\big|\text{Im}\,z\big|<\sigma$ of the complex $z$-plane, the errors in all three $\widehat{T}{}^{(s)}_n[f]$ are $O(e^{-2n\pi\sigma/T})$ as $n\to\infty$,
for all practical purposes.
 \end{abstract}

\vspace{1cm} \noindent {\bf Mathematics Subject Classification
2010:} 41A55, 65B15, 65D30, 65D32.

\vspace{1cm} \noindent {\bf Keywords and expressions:} Hadamard   Finite Part,
 supersingular integrals, numerical quadrature, trapezoidal rule, periodic integrands.
\thispagestyle{empty}
\newpage
\pagenumbering{arabic}

\section{Introduction and background} \label{se1}
  \setcounter{equation}{0} \setcounter{theorem}{0}
 Let
 \be\label{eq1} I[f]=\intBar^b_a f(x)\,dx,\quad f(x)=\frac{g(x)}{(x-t)^3}, \quad g\in C^\infty,\quad t\in(a,b)\ \text{fixed}.\ee
 $\intBar^b_a f(x)\,dx$ denotes  the {\em Hadamard Finite Part (HFP)} of the {\em supersingular} integral $\int^b_af(x)\,dx$, which does not exist in the regular sense due to  the term $(x-t)^{-3}$.

 Using a generalization of the classical Euler--Maclaurin  expansion by Sidi \cite{Sidi:2012:EME-P1},\footnote{The generalization of the  Euler--Maclaurin  expansion  of  \cite{Sidi:2012:EME-P1} concerns $\intBar^b_a u(x)\,dx$, $u\in C^\infty(a,b)$, when $u(x)$ has the asymptotic expansions
 $$
u(x)\sim K(x-a)^{-1}+\sum^{\infty}_{s=0}c_s\,(x-a)^{\gamma_s}
\quad \text{as}\   x\to a+,$$
$$
u(x)\sim L(b-x)^{-1}+\sum^{\infty}_{s=0}d_s\,(b-x)^{\delta_s}
\quad \text{as}\  x\to b-,$$
$$\gamma_s\ \text{distinct and arbitrary},\quad
\gamma_s\neq -1\quad \forall  s;\quad  \text{Re\,}\gamma_0\leq
\text{Re\,}\gamma_1\leq\text{Re\,}\gamma_2\leq\cdots;\quad
\lim_{s\to\infty}\text{Re\,}\gamma_s=+\infty,$$
$$ \delta_s\ \text{distinct and arbitrary},\quad\delta_s\neq -1\quad \forall s;\quad \text{Re\,}\delta_0\leq\text{Re\,}\delta_1\leq
\text{Re\,}\delta_2\leq\cdots;\quad
\lim_{s\to\infty}\text{Re\,}\delta_s=+\infty.$$
Then, with $h=(b-a)/n$ and with Euler's constant $C=0.577\cdots$, there holds
\begin{align*}
h\sum^{n-1}_{j=1}u(a+jh)\sim\intBar^b_au(x)\,dx&+ K(C-\log h)+
\sum^{\infty}_{\substack{s=0\\ \gamma_s\not\in\{2,4,6,\ldots\}}}
c_s\,\zeta(-\gamma_s)\,h^{\gamma_s+1}\\
&+L(C-\log h)+\sum^{\infty}_{\substack{s=0\\ \delta_s\not\in\{2,4,6,\ldots\}}}
d_s\,\zeta(-\delta_s)\,h^{\delta_s+1}\quad\text{as $n\to\infty$}.
\end{align*}}
 in a recent work  by Sidi \cite{Sidi:2019:SSI-P1}, the author developed three new  trapezoidal-like
 numerical quadrature formulas $\widehat{T}{}^{(s)}_n[f],$ $s=0,1,2,$ that have excellent convergence properties for functions $f(x)$ that are such that
 \be\label{eq2}   f(x)\ \ \text{$T$-periodic}, \quad
 f\in C^\infty (\mathbb{R}_t),\quad T=b-a,\quad\mathbb{R}_t=\mathbb{R}\setminus\{t\pm kT\}^\infty_{k=0}.\ee
 With $h=T/n$, these formulas read
 \begin{align}
 \widehat{T}{}^{(0)}_n[f]&=h\sum^{n-1}_{j=1}f(t+jh)
          -\frac{\pi^2}{3}\,g'(t)\,h^{-1}+\frac{1}{6}\,g'''(t)\,h, \label{eqT0}\\
  \widehat{T}{}^{(1)}_n[f]&=h\sum^n_{j=1}f(t+jh-h/2)
                      -\pi^2\,g'(t)\,h^{-1},\label{eqT1} \\
 \widehat{T}{}^{(2)}_n[f]&=2h\sum^n_{j=1}f(t+jh-h/2)-
              \frac{h}{2}\sum^{2n}_{j=1}f(t+jh/2-h/4).\label{eqT2}
 \end{align}
Theorem 5.1 in \cite{Sidi:2019:SSI-P1} states that,
provided  $f(x)$ is as in \eqref{eq1}--\eqref{eq2},
  $\widehat{T}{}^{(s)}_n[f]\to I[f]$ as $n\to\infty$ with spectral accuracy; that is,
\be\label{eq3}
\widehat{T}{}^{(s)}_n[f]-I[f]=O(n^{-\mu})\quad \text{as $n\to\infty$} \quad \forall \mu>0.\ee

Supersingular  integrals arise in different areas of science and engineering, and several numerical quadrature formulas for computing them exist in the literature. We do not intend to review them here; instead, we refer the reader to the
bibliography of \cite{Sidi:2019:SSI-P1} for some of the related literature.

The main purpose of this work is two-fold:
(i)\,to explore the exactness properties of the quadrature formulas
$\widehat{T}{}^{(s)}_n[f]$ and (ii)\,to expand on  the convergence properties of the $\widehat{T}{}^{(s)}_n[f]$   when $f(z)$ is $T$-periodic and analytic in a strip of the complex $z$-plane that includes the real axis, with poles of order three at $x=t+kT$, $k=0,\pm1,\pm2,\ldots;$ we aim at improving and refining \eqref{eq3} considerably.

The integrands we will be working with in the sequel are of the special form
\be \label{eqKer}
f(x)=K(x)u(x),\quad K(x)\equiv\frac{\cos\frac{\pi(x-t)}{T}}{\sin^3\frac{\pi(x-t)}{T}},\quad u\in C^\infty[a,b].\ee
Such integrands arise naturally when computing Cauchy transforms on the unit circle, for example.
Throughout this work, we treat $t$ as a fixed parameter and not as a variable.

The paper is organized as follows:
In Section \ref{se2},   we provide the statements of  three  theorems that concern
the exactness properties of the quadrature formulas $\widehat{T}{}^{(s)}_n[f]$. Of these, Theorem \ref{thw1} provides the eigenvalues of  the kernel
$K(x)={\cos\frac{\pi(x-t)}{T}}/{\sin^3\frac{\pi(x-t)}{T}}$; the proof of this theorem is given Section \ref{se3}. Theorems \ref{thw2} and \ref{thw3} concern the application of the quadrature formulas $\widehat{T}{}^{(s)}_n[f]$ on the eigenfunctions of $K(x)$ and shows that they preserve  some of the eigenvalues; the proofs of these theorems are given in Section \ref{se4}.
In Section \ref{se5}, we develop the subject of the convergence of the quadrature formulas as they are applied to $T$-periodic integrands $f(x)=K(x)u(x)$ when $u(z)$  is also  analytic in a strip of the complex $z$-plane that contains the real axis. The main result of this development is stated as Theorem \ref{thw6}, whose proof is given in Section \ref{se6}.

\section{Exactness property of the $\widehat{T}{}^{(s)}_n[f]$}\label{se2}
\setcounter{equation}{0} \setcounter{theorem}{0}
It is well-known that the  trapezoidal rule for regular integrals has an interesting exactness property, as stated in Theorem \ref{th0}:

\begin{theorem} \label{th0} Let $I[f]=\int^b_a f(x)\,dx$ be a regular integral. Then the trapezoidal rule approximation for $I[f]$, namely,
 \be\label{eqtr1} Q_n[f]=h\bigg[\frac{1}{2}f(a)+\sum^{n-1}_{j=1}f(a+jh)+\frac{1}{2}f(b)\bigg];\quad h=\frac{b-a}{n},\quad n\  \text{integer},\ee
 is exact  when $f(x)$ is a trigonometric
  polynomial of degree at most $n-1$ with period $T=b-a$. That is,
 \be\label{eqtr2} Q_n[f]=I[f]\quad \forall f(x)=\sum^{n-1}_{m=-(n-1)}c_me^{\mrm{i}\,2m\pi x/T}.\ee
\end{theorem}

We showed  in Sidi \cite[Theorems 5.1 and 10.1]{Sidi:2013:CNQ} that the numerical quadrature formulas developed there for periodic  {\em Cauchy Principal Value} integrals and  {\em hypersingular} integrals, when applied to $\intbar^b_a \cot\frac{\pi(x-t)}{T} u(x)\,dx$ and
 $\intBar^b_a \csc^2\frac{\pi(x-t)}{T} u(x)\,dx$, respectively,   also enjoy  interesting exactness properties when $u(x)$ is a trigonometric polynomial of period $T=b-a$.

Here we  show that each of the numerical quadrature formulas  $\widehat{T}{}^{(s)}_n[f]$ given in \eqref{eqT0}--\eqref{eqT2}  for  supersingular integrals $I[f]$ with $T$-periodic $f(x)$,
enjoy  similar exactness properties as described in Theorems \ref{thw1} and \ref{thw2}, which form two of the main results of this work.

\begin{theorem} \label{thw1}
With $T=b-a$, let
\be\label{eqIfmc} e_m(x)=e^{\mrm{i}2m\pi x/T},\quad f_m(x)=K(x)e_m(x),\quad m=0,\pm1,\pm2,\ldots.\ee
Then the  supersingular integral $I[f_m]=\intBar^b_af_m(x)\,dx$ satisfies
\be\label{eqIfm} I[f_m]=-\mrm{i}\,\text{\em sgn}(m)2Tm^2e_m(t),
\quad m=0,\pm1,\pm2,\ldots.
\ee
{\em [}Note that $f_m(x)$ is $T$-periodic and has a supersingularity of the  form $(x-t)^{-3}$ at $x=t$.{\em ]}
\end{theorem}

\noindent{\bf Remark:} Observe  that this theorem  actually states that $e_m(x)$ are actually  eigenfunctions of the kernel
$K(x)={\cos\frac{\pi(x-t)}{T}}/{\sin^3\frac{\pi(x-t)}{T}}$  with corresponding eigenvalues $-\mrm{i}\,\text{sgn}(m)2Tm^2$, $m=0,\pm1,\pm2,\ldots.$

\begin{theorem} \label{thw2}
With $e_m(x)$ and $f_m(x)$ as in Theorem \ref{thw1},
the quadrature formulas $\widehat{T}{}^{(s)}_n[f_m]$  satisfy the following:

\begin{align} \label{eqTn8} \widehat{T}{}^{(0)}_n[f_m]&=\mrm{i}\,\frac{T}{n}\,
\bigg( B_{m,n}-\frac{2}{3}mn^2-\frac{4}{3}m^3\bigg)\,e_m(t), \\
\widehat{T}{}^{(1)}_n[f_m]&=\mrm{i}\,\frac{T}{n}\,
\big( [B_{m,2n}-B_{m,n}]-2mn^2\big)\,e_m(t), \label{eqTn9}\\
\widehat{T}{}^{(2)}_n[f_m]&=\mrm{i}\,\frac{T}{n}\,
\bigg(2[B_{m,2n}-B_{m,n}]-\frac{1}{2}[B_{m,4n}-B_{m,2n}]\bigg)\,e_m(t), \label{eqTn10}
\end{align}
where $B_{m,n}$ are defined as follows:
\begin{enumerate}
\item For $m=0$, we have $ B_{0,n}=0$. For arbitrary $m$, there holds $B_{-m,n}=-B_{m,n}.$
\item
Given $m\geq0$,  let $k$ and $r$ be (unique) integers, $k\geq0$ and  $0\leq r\leq n-1$, such that $m=kn+r$. Then
\be\label{eqTn7}   B_{m,n}=B_{kn+r,n}=B_{r,n}=\frac{2}{3}\,rn^2-2r^2n+\frac{4}{3}\,r^3.\ee
\item Thus, $\big|B_{m,n}\big|\leq \max_{0\leq i\leq n-1}\big|B_{i,n}\big|$
independent of  $m$, hence $\{B_{m,n}\}^\infty_{m=-\infty}$ is a bounded sequence for each fixed $n$.
\end{enumerate}
\end{theorem}

\begin{theorem}\label{thw3} All three  quadrature formulas
 $\widehat{T}{}^{(s)}_n[f]$ possess the exactness property that
\be \label{eqTn8a} \widehat{T}{}^{(s)}_n[f_m]=I[f_m],\quad m=0,\pm1,\ldots,\pm(n-1),\ee
hence  that
\be\label{eqTn7a} \widehat{T}{}^{(s)}_n[f]=I[f],\quad
f(x)= \frac{\cos\frac{\pi(x-t)}{T}}{\sin^3\frac{\pi(x-t)}{T}}\, u(x)\quad \forall u(x)=\sum^{n-1}_{m=-(n-1)}c_me^{\mrm{i}\,2m\pi x/T}. \ee
\end{theorem}

We provide the proofs of these theorems in the next two sections.

\section{Proof of Theorem \ref{thw1}} \label{se3}
\setcounter{equation}{0} \setcounter{theorem}{0}

We start by noting that
 $$ I[f_m]=\bigg[\intBar^b_a\frac{\cos\frac{\pi(x-t)}{T}}
 {\sin^3\frac{\pi(x-t)}{T}}\,e^{\mrm{i}\,2m\pi(x- t)/T}dx\bigg]\,
e^{\mrm{i}\,2m\pi t/T}.$$
Making  the variable transformation $y=2\pi(x-t)/T$ in the integral inside the square brackets,  and using the fact that the transformed integrand is $2\pi$-periodic, we obtain
\be \label{eq510}I[f_m]=\frac{T}{2\pi}A_{m} e^{\mrm{i}\,2m\pi t/T}, \quad A_{m}=\intBar^{\pi}_{-\pi}\frac{\cos(\tfrac{1}{2}y)}{\sin^3(\tfrac{1}{2}y)}\,e^{\mrm{i}my}\,dy.\ee

Next,
$$A_{m}=\intBar^{\pi}_{-\pi}\frac{\cos(\tfrac{1}{2}y)}{\sin^3(\tfrac{1}{2}y)}\,\cos(my)\,dy+
\mrm{i} \intBar^{\pi}_{-\pi}\frac{\cos(\tfrac{1}{2}y)}{\sin^3(\tfrac{1}{2}y)}\,\sin(my)\,dy,$$ and
since $\intBar^{\pi}_{-\pi}[{\cos(\tfrac{1}{2}y)}/{\sin^3(\tfrac{1}{2}y)}]\,\cos(my)\,dy=0$  due to its integrand being odd, it follows that
\be \label{eq511}A_{m}=\mrm{i} \intBar^{\pi}_{-\pi}\frac{\cos(\tfrac{1}{2}y)}{\sin^3(\tfrac{1}{2}y)}\,\sin(my)\,dy
\quad \Rightarrow\quad A_0=0,\quad A_{-m}=-A_m.\ee
Therefore, it is sufficient to study $A_m$ only for nonnegative $m$, and this is what we do in the sequel.
Now,
\begin{align*} \sin[(m+1)y]+ \sin[(m-1)y]&=2 \sin(my)\cos y \notag\\
&=2[1-2\sin^2(\tfrac{1}{2}y)]\sin(my),\end{align*}
from which, by multiplying  by $\cos(\tfrac{1}{2}y)/\sin^3(\tfrac{1}{2}y)$, we obtain the identity

\be\label{eq513}\frac{\cos(\tfrac{1}{2}y)}{\sin^3(\tfrac{1}{2}y)}
\big(\sin[(m+1)y]- 2 \sin(my)+ \sin[(m-1)y]\big)=
-4\cot(\tfrac{1}{2}y)\,\sin(my). \ee
Upon integrating  both sides of   this identity over $(-\pi,\pi)$ and invoking \eqref{eq511}, we obtain
\be\label{eqLmn11} A_{m+1}-2A_{m}+A_{m-1}=-4\mrm{i}\int^\pi_{-\pi}\cot(\tfrac{1}{2}y)\sin(my)\,dy.\ee
[Note that the integral on the right-hand side of \eqref{eqLmn11} is defined in the regular sense.]
 By making the variable transformation $y=2z$ in this integral, and  invoking Gradshteyn and Ryzhik \cite[p. 391, formula 3.612(7)]{Gradshteyn:2007:TIS},  we obtain

$$ \int^\pi_{-\pi}\cot(\tfrac{1}{2}y)\sin(my)\,dy=4\int^{\pi/2}_0{\cos z}\,\frac{\sin(2mz)}{\sin z}\,dz=2\pi, \quad  m=1,2,\ldots .$$
Substituting this in \eqref{eqLmn11}, we obtain the following  recursion relation for the $A_m$:
$$ A_{m+1}-2A_{m}+A_{m-1}=-8\pi\mrm{i},\quad m=1,2,\ldots.$$
It is easy to see that the general solution of this recursion relation for $A_m$ is of the form
$$A_m=\alpha + \beta m-\mrm{i}4\pi m^2,\quad \text{ $\alpha$ and $\beta$   constants to be determined.}$$
First,
$A_0=0$ by \eqref{eq511}; this gives $\alpha=0$. Next, letting $m=1$ in the integral representation of $A_m$ in \eqref{eq511}, and simplifying the integrand, we obtain
$$A_1=\mrm{i} \intBar^{\pi}_{-\pi}\frac{\cos(\tfrac{1}{2}y)}{\sin^3(\tfrac{1}{2}y)}\,\sin y\,dy=
2\mrm{i} \intBar^{\pi}_{-\pi}\bigg[\frac{1}{\sin^2(\tfrac{1}{2}y)}-1\bigg]\,dy,$$ which,
 by the fact that (see \cite[Appendix A]{Sidi:2013:CNQ}, for example) $\intBar^{\pi}_{-\pi}{\csc^2(\tfrac{1}{2}y)}\,dy=0,$ gives $A_1=-4\pi\mrm{i},$ which in turn  implies $\beta=0$. Consequently, taking into account that $A_{-m}=-A_m$, we have
$$ A_m=-\mrm{i}\,\text{sgn}(m)\,4\pi m^2, \quad m=0,\pm1,\pm2,\ldots,$$
which, upon substituting into \eqref{eq510}, gives \eqref{eqIfm}.

\section{Proofs of Theorems  \ref{thw2}  and \ref{thw3}} \label{se4}
 \setcounter{equation}{0} \setcounter{theorem}{0}
\subsection{Preliminaries}
By the fact that $g(x)=(x-t)^3f(x)$ in \eqref{eq1}, we  realize that we must first address the issue of determining $g(x)$ and its  first three derivatives at $x=t$  when  $f(x)$ is of the form,
\be\label{eqpqs}f(x)= \frac{\cos\frac{\pi(x-t)}{T}}{\sin^3\frac{\pi(x-t)}{T}}\, u(x)
 \quad \Rightarrow\quad  g(x)=(x-t)^3\frac{\cos\frac{\pi(x-t)}{T}}{\sin^3\frac{\pi(x-t)}{T}}\, u(x).\ee
We achieve this by expanding $g(x)$ in a Taylor series about $x=t$.
We start by realizing that
$$ \frac{\cos z}{\sin^3 z}= \frac{1}{z^3}\frac{1-\tfrac{1}{2}z^2+O(z^4)}{1-\tfrac{1}{2}z^2+O(z^4)}=
\frac{1}{z^3}\big[1+O(z^4)\big]\quad\text{as $z\to0$.}$$
Using this in \eqref{eqpqs} and also expanding  $u(x)$ about $x=t$, we obtain
$$ g(x)=\bigg(\frac{T}{\pi}\bigg)^3\bigg[\sum^3_{i=0}\frac{u^{(i)}(t)}{i!}(x-t)^i
+O\big( (x-t)^4\big)\bigg]\quad \text{as $x\to t$,}$$
which implies that
\be\label{eqrty} g^{(i)}(t)=\bigg(\frac{T}{\pi}\bigg)^3 u^{(i)}(t),\quad i=0,1,2,3.\ee

Next, letting
\be \label{ttilde} \widetilde{T}_n[f]=h\sum^{n-1}_{j=1}f(t+jh),\ee
we rewrite \eqref{eqT0}--\eqref{eqT2} in the form
(see \cite[Section 4]{Sidi:2019:SSI-P1})
\begin{align}
 \widehat{T}{}^{(0)}_n[f]&=\widetilde{T}_n[f]
          -\frac{\pi^2}{3}\,g'(t)\,h^{-1}+\frac{1}{6}\,g'''(t)\,h, \quad
         \label{eqT0p}\\
  \widehat{T}{}^{(1)}_n[f]&=(2\widetilde{T}_{2n}[f]-\widetilde{T}_n[f])
                      -\pi^2\,g'(t)\,h^{-1},\label{eqT1p} \\
 \widehat{T}{}^{(2)}_n[f]&=2(2\widetilde{T}_{2n}[f]-\widetilde{T}_n[f])
 -(2\widetilde{T}_{4n}[f]-\widetilde{T}_{2n}[f]). \label{eqT2p}
 \end{align}
 This suggests that we can unify and shorten the proofs for the three $\widehat{T}{}^{(s)}_n[f_m]$ since we only have to analyze $\widetilde{T}_n[f_m]$ in detail. We do this in Theorem \ref{thBmn} that follows.

\subsection{Analysis of $\widetilde{T}_n[f_m]$}

 \begin{theorem}\label{thBmn} $\widetilde{T}_n[f_m]$ satisfies
\be\label{eqpq}
\widetilde{T}_n[f_m]=\bigg(\mrm{i}\,\frac{T}{n}\,B_{m,n}\bigg)e_m(t),
 \quad m=0,\pm1,\pm2,\ldots,\ee
where $B_{m,n}$ has the following properties:
\begin{gather}
 B_{-m,n}=-B_{m,n}\ \forall m\quad \Rightarrow\quad B_{0,n}=0, \label{eqmm1}\\
B_{m,n}=\frac{2}{3}mn^2-2m^2n+\frac{4}{3}m^3,\quad m=1,\ldots n-1, \label{eqmm2}\\
B_{m,n}=\text{\em sgn}(m)B_{kn+r,n}=\text{\em sgn}(m)B_{r,n}\ \forall m;\quad k\geq0,\ r\in \{0,1,\ldots,n-1\},\label{eqmm3}\end{gather}
where $k$ and $r$ are unique integers for which $\big|m\big|=kn+r$.
\end{theorem}
\begin{proof}
We start by observing that, by \eqref{ttilde},
\be \widetilde{T}_n[f_m]=\mrm{i}\,\frac{T}{n}B_{m,n},\quad
B_{m,n}=-\mrm{i}\sum^{n-1}_{j=1}\frac{\cos(\tfrac{1}{2}y_j)}{\sin^3(\tfrac{1}{2}y_j)}\,e^{\mrm{i}my_j},
\quad y_j=\frac{2j\pi}{n},\quad j=1,2,\ldots.\label{eqpq3}\ee
Now
  $$ B_{m,n}=-\mrm{i}\sum^{n-1}_{j=1}\frac{\cos(\tfrac{1}{2}y_j)}{\sin^3(\tfrac{1}{2}y_j)}\,\cos(my_j)
 +\sum^{n-1}_{j=1}\frac{\cos(\tfrac{1}{2}y_j)}{\sin^3(\tfrac{1}{2}y_j)}\,\sin(my_j).$$
Because $y_{n-j}=2\pi-y_j$, $j=1,\ldots,n-1,$ we have
$$ \frac{\cos(\tfrac{1}{2}y_{n-j})}{\sin^3(\tfrac{1}{2}y_{n-j})}\,\cos(my_{n-j})=
-\frac{\cos(\tfrac{1}{2}y_j)}{\sin^3(\tfrac{1}{2}y_j)}\,\cos(my_j), \quad j=1,\ldots,n-1,$$
and since
$\sum^{n-1}_{j=1}w_{n-j}=\sum^{n-1}_{j=1}w_j$, we have
$$\sum^{n-1}_{j=1}\frac{\cos(\tfrac{1}{2}y_j)}{\sin^3(\tfrac{1}{2}y_j)}\,\cos(my_j)=0.$$
As a result,
\be\label{eqBmn}B_{m,n}=\sum^{n-1}_{j=1}\frac{\cos(\tfrac{1}{2}y_j)}{\sin^3(\tfrac{1}{2}y_j)}\,\sin(my_j)\quad
\Rightarrow\quad B_{0,n}=0,\quad B_{-m,n}=-B_{m,n},\ee
hence \eqref{eqmm1} is proved.
Therefore, it is sufficient to study $B_{m,n}$ only for positive $m$.

Next,  for every $m\geq 0$, there exist unique integers $k$ and $r$, $k\geq 0$ and  $0\leq r\leq n-1$, such that $m=nk+r$.
(Thus,  $k=0$ and $r=m$ for $0\leq m\leq n-1$, while $k=1$ and $r=0$ for $m=n$.)
By the fact that
 $$\sin[(kn+r)y_j]=\sin(2kj\pi+ry_j)=\sin(ry_j),\quad r=0,\ldots,n-1,$$
 we realize that
\be\label{eq5181} B_{m,n}=B_{kn+r,n}=B_{r,n}\quad \text{when $m\geq0$},\ee
which, upon combining with \eqref{eqBmn}, results in \eqref{eqmm3}.
Thus, we need to concern ourselves only with $1\leq m\leq n-1$ since $k=0$ and
$r=m$ in such a case, and this is what we do in the sequel.

We start by deriving a recursion relation for the $B_{m,n}$  analogous to that for the $A_m$ given in \eqref{eqLmn11}. Replacing $y$ in \eqref{eq513} by $y_j$ and summing over $j$, we obtain
\be\label{eq513c} B_{m+1,n}-2B_{m,n}+B_{m-1,n}=-4C_{m,n},\quad C_{m,n}=
\sum^{n-1}_{j=1}\cot(\tfrac{1}{2}y_j)\,\sin(my_j).\ee
To determine $C_{m,n}$, we proceed as follows: First,
\be\label{eq5194} \sin(my\pm \tfrac{1}{2}y)=\sin(my)\cos(\tfrac{1}{2}y)\pm \cos(my)\sin(\tfrac{1}{2}y).\ee
Dividing both sides of this identity by $\sin(\tfrac{1}{2}y)$, replacing $y$ by $y_j$, and
summing over $j$, we obtain
\be \label{eq5199} \sum^{n-1}_{j=1}\frac{\sin(my_j\pm \tfrac{1}{2}y_j)}
{\sin(\tfrac{1}{2}y_j)}=C_{m,n}\pm \sum^{n-1}_{j=1}\cos(my_j).\ee
Now,
$$ \sum^{n-1}_{j=1}\cos(my_j)=\text{Re}\sum^{n-1}_{j=1}e^{\mrm{i}my_j}=
\text{Re}\sum^{n-1}_{j=1}\big(e^{\mrm{i}2m\pi/n}\big)^j=-1,\quad m=1,\ldots,n-1,$$
since $e^{\mrm{i}2m\pi/n}\neq1$ for $m=1,\ldots,n-1.$
Upon also defining
$$ D_{k,n}=\sum^{n-1}_{j=1}\frac{\sin(ky_j- \tfrac{1}{2}y_j)}
{\sin( \tfrac{1}{2}y_j)},$$ \eqref{eq5199} gives the equalities
$$ D_{m,n}=C_{m,n}+1\quad\text{and}\quad D_{m+1,n}=C_{m,n}-1,\quad m=1,\ldots,n-1.$$
Eliminating $C_{m,n}$, we obtain
$$ D_{m+1,n}=D_{m,n}-2,\quad m=1,\ldots,n-1, $$
which, upon realizing that $D_{1,n}=n-1$, gives
$$ D_{m,n}=n-2m+1 \quad \Rightarrow \quad C_{m,n}=n-2m, \quad m=1,\ldots,n-1.$$
As a result, \eqref{eq513c} becomes
\be\label{eq512a} B_{m+1,n}-2B_{m,n}+B_{m-1,n}=8m-4n,\quad  m=1,\ldots,n-1.\ee
It is easy to see that the general solution of this recursion relation for $B_{m,n}$ is of the form
\be \label{eqBm1}B_{m,n}=\alpha + \beta m-2nm^2+\frac{4}{3}m^3,\quad m\geq 1,\ee
 $\alpha$ and $\beta$   being constants to be determined.
They can be obtained by invoking the values of $B_{1,n}$ and $B_{2,n}$.

We start with $B_{1,n}$. Letting $m=1$ in  \eqref{eqBmn} and  simplifying, we obtain
\be \label{En1} B_{1,n}
=2\sum^{n-1}_{j=1}\bigg(\frac{1}{\sin^2(\tfrac{1}{2}y_j)}-1\bigg)=2(L_n-n+1),\quad
L_n= \sum^{n-1}_{j=1}\frac{1}{\sin^2(\tfrac{1}{2}y_j)}.\ee
To determine $L_n$, we proceed as follows: We first express $L_n$ in the form
\be \label{En2} L_n=
\sum^{n-1}_{j=1}\frac{1}{1-\eta_j^2}=\frac{1}{2}\sum^{n-1}_{j=1}\bigg(
\frac{1}{1-\eta_j}+\frac{1}{1+\eta_j}\bigg)=\sum^{n-1}_{j=1}
\frac{1}{1-\eta_j},\quad \eta_j=\cos\bigg(\frac{j\pi}{n}\bigg).\ee
Here we have invoked  $\eta_{n-j}=-\eta_j$ and $\sum^{n-1}_{j=1}w_j=\sum^{n-1}_{j=1}w_{n-j}$.
Now, $\eta_1,\ldots,\eta_{n-1}$ are the points of extremum of the $n$th Chebyshev polynomial $T_n(z)$ in $(-1,1)$, hence the zeros of its derivative $T_n'(z)$.
Thus,
$$ \frac{T_n''(z)}{T_n'(z)}=\sum^{n-1}_{j=1}\frac{1}{z-\eta_j},$$ hence
\be \label{En3} L_n=\frac{T_n''(1)}{T_n'(1)}= \frac{n^2-1}{3}.\ee
(See, Rivlin \cite[p. 38]{Rivlin:1990:CP}, for example.)
Consequently,
$$ B_{1,n}=2\,\bigg(\frac{n^2-1}{3}-n+1\bigg).$$
As for $B_{2,n}$, letting $m=1$ in \eqref{eq512a} and recalling that $B_{0,n}=0$, we obtain
$$ B_{2,n}=2B_{1,n}+8-4n.$$
Substituting these values of $B_{1,n}$ and $B_{2,n}$ into \eqref{eqBm1}  (with $m=1$ and $m=2$ there), we obtain $\alpha=0$ and $\beta=2n^2/3$, hence \eqref{eqmm2}.
\end{proof}

\subsection{Completion of proofs}
With Theorem \ref{thBmn} available, we can now complete the proof of Theorem \ref{thw2}.
Since  $u(x)=e_m(x)$ when $f(x)=f_m(x)$, from \eqref{eqpqs} and \eqref{eqrty}, we have
$$g'(t)=\frac{T^3}{\pi^3}e_m'(t)=\mrm{i}\,2\frac{T^2}{\pi^2}m e_m(t)\quad\text{and}\quad
g'''(t)=\frac{T^3}{\pi^3}e_m'''(t)=-\mrm{i}\,8m^3e_m(t).$$
Substituting these in \eqref{eqT0p}--\eqref{eqT2p} and invoking also \eqref{eqpq}, we obtain
\eqref{eqTn8}--\eqref{eqTn10}. Finally, the expression given for $B_{m,n}$ in \eqref{eqTn7} is simply that in \eqref{eqmm2} proved in Theorem \ref{thBmn}.
This completes the proof of Theorem \ref{thw2}.

To complete  the proof of Theorem \ref{thw3}, we only need to verify  \eqref{eqTn8a}.
We can achieve this by substituting \eqref{eqmm2} in \eqref{eqTn8}--\eqref{eqTn10} and comparing with \eqref{eqIfm}. We leave the details to the reader.

\section{Convergence property of the $\widehat{T}{}^{(s)}_n[f]$}\label{se5}
\setcounter{equation}{0} \setcounter{theorem}{0}
It is well  known that the  trapezoidal rule $Q_n[f]$ in \eqref{eqtr1} converges exponentially in $n$ when applied to {\em regular} integrals $I[f]=\int^b_af(x)\,dx$
in case $f(z)$, as a function of the complex variable $z$,  is analytic in a strip of the $z$-plane containing the real axis and is  $(b-a)$-periodic   in this strip. The following theorem by Davis \cite{Davis:1959:NIP}, addresses this fully:

\begin{theorem}\label{thw5}
Let $f(z)$ be analytic and periodic with period $T=b-a$ in the infinite strip
$D_\sigma=\{z:\ \big|\text{\em Im}\, z\big|<\sigma\}$
 of the $z$-plane.  Then
\be\label{eq8} \big| Q_n[f]-I[f]\big|\leq T M(\tau)\frac{e^{-2n\pi\tau/T}}
{1-e^{-2n\pi\tau/T}}\quad \forall \tau\in(0,\sigma),\ee
{where}
\be\label{eq9} M(\tau)=\max_{x\in \mathbb{R}}\big|f(x+\mrm{i}\tau)\big|+\max_{x\in \mathbb{R}}\big|f(x-\mrm{i}\tau)\big|.\ee
\end{theorem}

We showed  in Sidi and Israeli \cite[Theorem 9]{Sidi:1988:QMP} and in  Sidi \cite[Theorems 6.1 and 6.2]{Sidi:2013:CNQ} that the numerical quadrature formulas developed in these papers  for periodic  {\em Cauchy Principal Value} integrals and  {\em hypersingular} integrals, when applied to $\intbar^b_a \cot\frac{\pi(x-t)}{T} u(x)\,dx$ and
 $\intBar^b_a \csc^2\frac{\pi(x-t)}{T} u(x)\,dx$, respectively,   also enjoy similar convergence properties when $f(z)$ is $T$-periodic and has  poles of order one and two, respectively, at the points $t+kT$, $k=0,\pm1,\pm2,\ldots,$ and is analytic in a strip of the $z$-plane containing the real axis.

Here we show that each of the numerical quadrature formulas  $\widehat{T}{}^{(s)}_n[f]$ given in \eqref{eqT0}--\eqref{eqT2}  for  supersingular integrals $I[f]$ with $T$-periodic $f(z)$
enjoys  similar convergence properties,
 as described in Theorem \ref{thw6}. The proof  of this theorem is provided in Section~\ref{se6}.

\begin{theorem}\label{thw6}
Let the function $u(z)$ be analytic and periodic with period $T=b-a$ in the infinite strip $D_\sigma=\{z:\ \big|\text{\em Im}\, z\big|<\sigma\}$
 of the $z$-plane, and let
 $$ f(x)=\frac{\cos\frac{\pi(x-t)}{T}}{\sin^3\frac{\pi(x-t)}{T}}\,u(x)\quad\text{and}\quad I[f]=\intBar^b_af(x)\,dx.$$
Define
 $E^{(s)}_n[f]=\widehat{T}{}^{(s)}_n[f]-I[f]$, $s=0,1,2.$ Then
   \begin{align}
  \big|E^{(0)}_n[f]\big|&\leq T M(\tau)\phi_n(\tau)\quad \forall \tau\in(0,\sigma), \label{eqEEn}\\
\big|E^{(1)}_n[f]\big|&\leq  T M(\tau)[\phi_n(\tau)+2\phi_{2n}(\tau)]\quad \forall \tau\in(0,\sigma), \label{eqT1qq}  \\
\big|E^{(2)}_n[f]\big|&\leq  T M(\tau)[2\phi_n(\tau)+5\phi_{2n}(\tau)+2\phi_{4n}(\tau)] \quad \forall \tau\in(0,\sigma).\label{eqT2qq}
 \end{align}
 where
\be M(\tau)=\max_{x\in \mathbb{R}}\big|F_1(x+\mrm{i}\tau)\big|+\max_{x\in \mathbb{R}}\big|F_1(x-\mrm{i}\tau)\big|,\quad
 \phi_n(\tau)=\frac{e^{-2n\pi\tau/T}}
{1-e^{-2n\pi\tau/T}}.
 \ee
Here $F_1(z)$ is $T$-periodic and analytic in the strip $D_\sigma$ and is given
as
 $$F_1(z)=K(z)\bigg[u(z)-u(t)-\frac{T}{\pi}u'(t)\tan\frac{\pi(z-t)}{T}
-\frac{T^2}{2\pi^2}u''(t)\sin^2\frac{\pi(z-t)}{T}\bigg].$$
\end{theorem}

\noindent{\bf Remark:} It is easy to see that, for all practical purposes, all three errors $E^{(s)}_n[f]$ are $O(e^{-2n\pi\sigma/T})$ as $n\to\infty$.
Of course, this improves the convergence result in \eqref{eq3} significantly for the supersingular integrals considered  here.

\section{Proof of Theorem \ref{thw6}} \label{se6}
\setcounter{equation}{0} \setcounter{theorem}{0}

We start by observing that, with $\widetilde{T}_n[f]$ as in \eqref{ttilde} and $\widehat{T}{}^{(0)}_n[f]$ as in \eqref{eqT0p}, we can reexpress  $\widehat{T}{}^{(1)}_n[f]$  in \eqref{eqT1p} and $\widehat{T}{}^{(2)}_n[f]$ in \eqref{eqT2p} as follows:
\begin{align}
\widehat{T}{}^{(1)}_n[f]&=2\widehat{T}{}^{(0)}_{2n}[f]-\widehat{T}{}^{(0)}_n[f],                  \label{eqT1q}  \\
 \widehat{T}{}^{(2)}_n[f]&= -2\widehat{T}{}^{(0)}_{4n}+5\widehat{T}{}^{(0)}_{2n}[f]-2\widehat{T}{}^{(0)}_n[f].
  \label{eqT2q}
 \end{align}
As a result, we also have
\begin{align}
E^{(1)}_n[f]&=2E^{(0)}_{2n}[f]-E^{(0)}_n[f],                  \label{eqT1qr}  \\
 E^{(2)}_n[f]&= -2E^{(0)}_{4n}[f]+5E^{(0)}_{2n}[f]-2E^{(0)}_n[f].  \label{eqT2qr}
 \end{align}
Clearly, this will help us  unify the treatments of  all three
$\widehat{T}{}^{(s)}_n[f]$, once we treat $\widehat{T}{}^{(0)}_n[f]$.

Next, following Yang \cite{Yang:2013:UAS}, we let
$$ u(z)=U_1(z)+U_2(z);\quad U_2(z)=u(t)+\frac{T}{\pi}u'(t)\tan\frac{\pi(z-t)}{T}
+\frac{T^2}{2\pi^2}u''(t)\sin^2\frac{\pi(z-t)}{T},$$
and $$f(z)=F_1(z)+F_2(z);\quad F_1(z)=K(z)\,U_1(z),\quad
 F_2(z)=K(z)\,U_2(z).$$
 Therefore,
 $$ I[f]=I[F_1]+I[F_2]\quad\text{and}\quad
  \widehat{T}{}^{(0)}_n[f]= \widehat{T}{}^{(0)}_n[F_1]+\widehat{T}{}^{(0)}_n[F_2].$$
Since $U_1(z)$ and $U_2(z)$ are both
$T$-periodic, so are $F_1(z)$ and $F_2(z)$.
 Expanding $U_2(z)$ about $z=t$ in a Taylor series, it is easy to verify that
 \be\label{eqUU}U_1(t)=U_1'(t)=U_1''(t)=0,\quad U_1'''(t)=u'''(t)-\frac{2\pi^2}{T^2}u'(t)\quad\Rightarrow\quad F_1(t)=\frac{T^3}{6\pi^3}U_1'''(t),\ee  which implies that $F_1(z)$ has no singularities in the strip $D_\sigma$ and that $I[F_1]$ is a regular integral, to which Theorem \ref{thw5} applies.

Let us now study $I[F_2]$ and $\widehat{T}{}^{(0)}_n[F_2]$.
We have
$$  I[F_2]=u(t)I_1+\frac{T}{\pi}u'(t)I_2+\frac{T^2}{2\pi^2}u''(t)I_3,$$
where, by Theorems \ref{thw1} and \ref{thw2},
$$I_1=I[Ke_0]=0=\widehat{T}{}^{(0)}_n[Ke_0]$$ and
since $\sin^2\tfrac{\pi(x-t)}{T}=\frac{1}{4}[2e_0(x)-e_2(x)e_{-2}(t)-e_{-2}(x)e_2(t)],$
$$ I_3=I[K\sin^2\tfrac{\pi(\cdot-t)}{T}]=0=
\widehat{T}{}^{(0)}_n[K\sin^2\tfrac{\pi(\cdot-t)}{T}] \quad \forall\ n\geq 3.$$
Next, it is known that
$$ I_2=I[K\tan\tfrac{\pi(\cdot-t)}{T}]=\intBar^b_a\frac{1} {\sin^2\frac{\pi(x-t)}{T}}\,dx=0.$$
As for $\widehat{T}{}^{(0)}_n[K\tan\tfrac{\pi(\cdot-t)}{T}]$, by
 \eqref{eqrty}--\eqref{eqT0p} and \eqref{En1}--\eqref{En3}, and by the fact that
 $\tan z=z+\tfrac{1}{3}z^3+O(z^5)$,
$$ \widetilde{T}{}^{(0)}_n[K\tan\tfrac{\pi(\cdot-t)}{T}]=\frac{T}{n}L_n\quad\Rightarrow\quad
\widehat{T}{}^{(0)}_n[K\tan\tfrac{\pi(\cdot-t)}{T}]=0.$$
We have thus shown that $I[F_2]=0=\widehat{T}{}^{(0)}_n[F_2].$
We conclude that $I[f]=I[F_1]$ and  $\widetilde{T}{}^{(0)}_n[f]=\widetilde{T}{}^{(0)}_n[F_1]$.

We now wish to show that
$\widehat{T}{}^{(0)}_n[f]=Q_n[F_1]$, where $Q_n[F_1]$, the trapezoidal rule approximation for $I[F_1]$, is given as
$$Q_n[F_1] =h\sum^{n-1}_{j=0}F_1(t+jh)=\widetilde{T}{}^{(0)}_n[F_1]+hF_1(t)$$
since $F_1(z)$ is $T$-periodic.
Therefore,  by \eqref{eqrty}--\eqref{eqT0p} and \eqref{eqUU},
$$ \widehat{T}{}^{(0)}_n[F_1]=Q_n[F_1]-hF_1(t)+\frac{T^3}{\pi^3}
\bigg(-\frac{\pi^2}{3}U_1'(t)h^{-1}+\frac{1}{6}U_1'''(t)h\bigg)=Q_n[F_1].$$

Combining everything, we have shown that $\widehat{T}{}^{(0)}_n[f]-I[f]=
Q_n[F_1]-I[F_1].$
We now complete the proof of \eqref{eqEEn} by applying Theorem \ref{thw5} to $I[F_1].$
The proofs of \eqref{eqT1qq} and \eqref{eqT2qq} are immediate.


\end{document}